\documentclass{amsart}
\usepackage{graphicx}

\newtheorem{theorem}{Theorem}[section]
\newtheorem{lemma}[theorem]{Lemma}
\newtheorem{cor}[theorem]{Corollary}

\theoremstyle{definition}

\theoremstyle{remark}
\newtheorem{remark}[theorem]{Remark}

\numberwithin{equation}{section}

\begin{document}

\title{Area-preserving geometric hermite interpolation}

\author{Geoffrey McGregor}
\address{Department of Mathematics, McGill University, Montreal, QC, Canada.}
\curraddr{805 Sherbrooke St W, Montreal, QC H3A 0B9, Canada.}
\email{Geoffrey.McGregor@mail.mcgill.ca}
\thanks{The research of GMc was supported in part by the Schulich Scholarship and the Murata Fellowship at McGill University.}

\author{Jean-Christophe Nave}
\address{Department of Mathematics, McGill University, Montreal, QC, Canada.}
\email{jcnave@math.mcgill.ca}
\thanks{The research of JCN was supported in part by the NSERC Canada Discovery Grants Program. Additionally, JCN would like to thank the Shanghai Jiaotong University Institute of Natural Sciences for hosting him while completing this work.}


\date{\today}



\begin{abstract}
 In this paper we establish a framework for planar geometric interpolation with exact area preservation using cubic B\'ezier polynomials. We show there exists a family of such curves which are $5^{th}$ order accurate, one order higher than standard geometric cubic Hermite interpolation. We prove this result is valid when the curvature at the endpoints does not vanish, and in the case of vanishing curvature, the interpolation is $4^{th}$ order accurate. The method is computationally efficient and prescribes the parametrization speed at endpoints through an explicit formula based on the given data. Additional accuracy (i.e. same order but lower error constant) may be obtained through an iterative process to find optimal parametrization speeds which further reduces the error while still preserving the prescribed area exactly.
\end{abstract}

\maketitle

\section{Introduction}
In the last several decades, interpolation using parametric curves has been studied extensively with the primary application being geometric design and computer graphics \cite{DeBoore, FarinBook,FaroukiPrac, Gregory, Hollig}. The class of parametric polynomials we are concerned with here are continuous in both the position and velocity vectors. These curves are classified as $G^1$, or geometric continuity of order 1. Further discussion of parametric and geometric continuity can be found in \cite{parsmooth}.

In this paper we introduce a $G^1$ cubic B\'ezier interpolation framework which, given a parametric curve $\langle \gamma(s) \, , \, \xi(s) \rangle \in \mathbb{R}^2$, with $s\in[s_0,s_1]$,  exactly matches the area $\int_{s_0}^{s_1}\xi(\tau)\gamma'(\tau)\text{d}\tau$. A conservative framework of this type may be of particular interest in areas of mathematics and engineering where physical laws must be obeyed, for examples see \cite{LeVeque}. In addition to being consistent with physical principles, conservative schemes offer additional stability properties \cite{AndyWan}. From the outset, conservative methods of this type may not directly appeal to the geometric design community, however the results presented in this paper show it performs surprisingly well when compared to standard non-conservative methods, such as the curvature matching method described in \cite{DeBoore} and the curvature variation energy method discussed in \cite{jaklicCVE}. One interesting finding of our work is that requiring a traditionally $4^{th}$ order cubic B\'ezier to exactly match the area of a target function results in a $5^{th}$ order accurate interpolating polynomial. In fact, we show that there is an entire family of area-preserving $G^1$ cubic polynomials which are $5^{th}$ order accurate. In many applications choosing any member of this family will provide sufficient accuracy, however, in situations where high precision is required, an optimal member of the family may be selected through an iterative procedure.  We discuss details of the optimization procedure and show that the error can be further reduced by several orders of magnitude. 
Before introducing the area-preserving method, we discuss a selection of relevant results from the geometric design and interpolation literature. 

Geometric interpolation was first introduced by de Boor et al. in \cite{DeBoore}. In that work a parametric-cubic interpolation framework matching function value, tangents and curvature at endpoints was derived. This is now classified as $G^2$ interpolation and is $6^{th}$ order accurate. The method reduces to $4^{th}$ order if the curvature vanishes anywhere within the interpolation domain.   Since that seminal result, work in geometric interpolation has grown extensively, with an emphasis on non-local quantities. For example, in \cite{ArcLength}, the author succeeded to create a $G^1$ interpolation framework which matches prescribed arclength using Pythagorean-hodograph (PH) curves. For more on PH curves see \cite{farouki3,farouki2,farouki1}. Other extensively studied methods of interpolation are concerned with seeking $G^1$ polynomial curves which minimize the strain energy, $\int_{t_0}^{t_1}(\kappa(t))^2\text{d}t$, where $\kappa(t)$ denotes the curvature, or the curvature variation energy $\int_{t_0}^{t_1}(\kappa'(t))^2\text{d}t$. For example, Jaklic and Zagar in \cite{jaklicSE} study $G^1$ cubic polynomials which minimize an approximate strain energy, also called the linearized bending energy. This is also studied in \cite{StrainEnergy} and extended to quintics in \cite{luG2SE}. In \cite{jaklicCVE} Jaklic and Zagar present a $4^{th}$ order accurate $G^1$ interpolation method which minimizes a functional approximating the curvature variation energy. Recently in \cite{CurvMin}, Lu et al. introduced a scheme which computes $G^1$ cubic interpolants minimizing the true curvature variation energy through a constrained minimization problem. The results are concluded to be better than the approximate methods in \cite{jaklicCVE}, however this additional accuracy comes with a significant increase in computational cost.  After an extensive search, it is to the best of our knowledge that area-preservation within the context of parametric interpolation is novel. In the results section we will compare our area-preserving $G^1$ interpolation method to the curvature matching method of de Boor et al. and the approximate curvature variation energy method of Jaklic and Zagar.

We begin section \ref{SecBezier} with a brief discussion of cubic B\'ezier curves and introduce the area-preserving cubic B\'ezier framework. We then prove two Lemmas which lead to the main Theorem stating that a certain class of area-preserving cubic B\'ezier curves are $5^{th}$ order accurate. The main Theorem is followed by a Corollary which states that no $6^{th}$ order method can be constructed within this framework. In section \ref{SecNum} numerical experiments are conducted to compare the area-preserving scheme to other geometric interpolation methods, while also verifying the results presented in section \ref{SecBezier}. We then discuss the optimized method and show that an increase in accuracy can be achieved by choosing an optimal member of the area-preserving cubic B\'ezier family. We conclude the numerical experiments section by discussing the application of area-preserving parametric interpolation to scalar conservation laws in one space dimension. This brief discussion solidifies the importance of the framework presented in this paper. In section \ref{Discussion} we give some concluding remarks and discuss future directions of work.

\section{Area-preserving cubic B\'ezier curves} \label{SecBezier}
A parametric cubic B\'ezier curve $\vec{B}(t)=\langle B_1(t)\, , \, B_2(t) \rangle$ interpolating point $\vec{A}$ to point $\vec{D}$, for $\vec{A},\vec{D} \in \mathbb{R}^2$,  is given by
\begin{equation}
\vec{B}(t)=\vec{A}(1-t)^3 +3\vec{C}_1(1-t)^2t+3\vec{C}_2(1-t)t^2+\vec{D}t^3, \quad \text{for $t\in[0,1]$}, \label{Bezier1}
\end{equation}
where $\vec{C}_1$ and $\vec{C}_2$ are referred to as control points which dictate the tangent direction and magnitude of $\vec{B}$ at $t=0$ and $t=1$. The tangent direction at endpoints will be extracted from the prescribed data, therefore, the remaining degrees of freedom are the magnitudes, $r_1>0$ and $r_2>0$, for the left and right tangents respectively. Taking the tangent direction at the left endpoint to be $\vec{\alpha}$ and $\vec{\beta}$ on the right we therefore obtain
\begin{equation*}
\vec{B}'(0)=r_1\vec{\alpha} \quad \text{and},\quad \vec{B}'(1)=r_2\vec{\beta} .
\end{equation*}
Rewriting the control points in terms of $\vec{A},\vec{D}, \vec{\alpha}, \vec{\beta}, r_1$ and $r_2$, yields the two-parameter family of cubic B\'ezier curves matching function value and tangent directions,
\begin{equation}
\vec{B}(t)=\vec{A}(1-t)^3 +3\Big(\vec{A}+\frac{r_1\vec{\alpha}}{3}\Big)(1-t)^2t+3\Big(\vec{D}-\frac{r_2\vec{\beta}}{3}\Big)(1-t)t^2+\vec{D}t^3, \quad \text{for $t\in[0,1]$}. \label{Bezier2}
\end{equation}
Given a parametric curve $\langle \gamma(s) \, , \, \xi(s)\rangle$ parametrized by $s \in [s_0,s_1]$, we reduce the two-parameter family of solutions (\ref{Bezier2}) to a one-parameter family by imposing the the parametric area condition
\begin{equation}
\int_{0}^{1}B_2(t)B_1'(t)\text{d}t=\int_{s_0}^{s_1}\xi(s)\gamma'(s)\text{d}s\label{BezArea}.
\end{equation}
Using that B\'ezier curves are formed by generalized convex combinations, we simplify the expression by shifting $\vec{A}$ to the origin, which yields
\begin{align}
\vec{B}(t)-\vec{A}=&\vec{B}(t)-\vec{A}((1-t)^3+3(1-t)^2t+3(1-t)t^2+t^3) \nonumber\\
=&r_1\vec{\alpha}(1-t)^2t+3\Big(\vec{D}^*-\frac{r_2\vec{\beta}}{3}\Big)(1-t)t^2+\vec{D}^*t^3, \quad \text{for $t\in[0,1]$}, \label{Bezier3}
\end{align} 
where $\vec{D}^*=\vec{D}-\vec{A}$, but for convenience we will drop the $^*$. Using the shifted form (\ref{Bezier3}), the integrand of (\ref{BezArea}), $B_2(t)B_1'(t)$, is a $5^{th}$ degree polynomial with coefficients given by $r_1,r_2, \vec{\alpha}, \vec{\beta}$ and $\vec{D}$. Integrating each term and simplifying the result yields
\begin{equation}
\int_{0}^{1}B_2(t)B_1'(t)\text{d}t=\frac{r_1r_2}{60}(\vec{\alpha}\times\vec{\beta} )+\frac{r_1}{10}(\vec{D}\times\vec{\alpha})+\frac{r_2}{10}(\vec{\beta}\times\vec{D})+\frac{D_1D_2}{2}=\mathcal{C},\label{BezArea2}
\end{equation}
where $\mathcal{C}$ is the prescribed area after shifting the left endpoint to the origin, $D_1$ and  $D_2$ are the first and second components of $\vec{D}$ respectively, and the notation ``$\times$'' refers to the planar vector product $\vec{\alpha}\times \vec{\beta}=\alpha_1\beta_2-\beta_1\alpha_2$. Note that shifting $\langle \gamma(s) \, , \, \xi(s)\rangle$ to $\langle \gamma(s) - x_0, \xi(s)-y_0 \rangle$, yields $\mathcal{C}=\int_{s_0}^{s_1}\xi(s)\gamma'(s)\text{d}s-y_0(\gamma(s_1)-\gamma(s_0))$. Returning to equation (\ref{BezArea2}), and moving $\frac{D_1D_2}{2}$ to the right hand side, we obtain the area constraint equation

\begin{equation}
\frac{r_1r_2}{60}(\vec{\alpha}\times\vec{\beta} )+\frac{r_1}{10}(\vec{D}\times\vec{\alpha})+\frac{r_2}{10}(\vec{\beta}\times\vec{D})=\mathcal{C}_R,\label{AreaCon}
\end{equation}
prescribing the signed area about the secant line of the parametric polynomial $\vec{B}$ to equal that of the data. A sketch of the prescribed signed area is given in Figure \ref{Relative area}.
\begin{figure}[tb]
\begin{center}
\includegraphics[width=50mm,height=40mm]{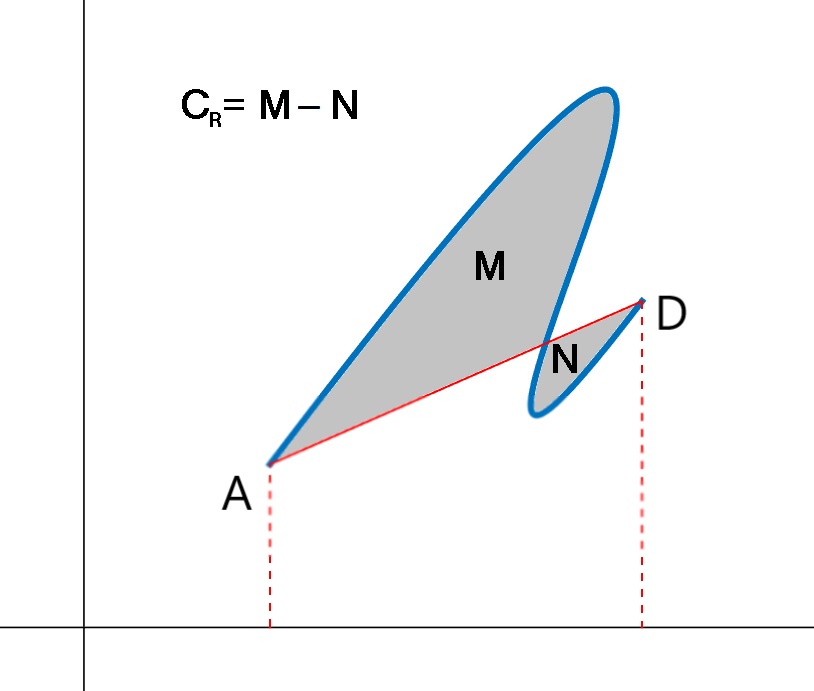}
\end{center}
\caption{ Signed area about the secant line between $\vec{A}$ and $\vec{D}$. }
\label{Relative area}
\end{figure}

Our goal is as follows: Given the signed area $\mathcal{C}_R$, we aim to find choices of $r_1>0$ and $r_2>0$ which satisfy equation (\ref{AreaCon}), and in addition, obtain an estimate for the convergence rate of the error in a suitable norm as the size of interpolation domain, $|\vec{D}|$, tends to zero.

We begin by discussing existence of solutions to (\ref{AreaCon}) by looking at the three possible cases, $\mathcal{C}_R>0,\mathcal{C}_R=0$ and $\mathcal{C}_R<0$. Existence reduces to investigating the signs of the coefficients in the left hand side of equation (\ref{AreaCon}). If $\mathcal{C}_R>0$, then at least one of the coefficients $(\vec{\alpha}\times\vec{\beta}), (\vec{D}\times\vec{\alpha})$, or $(\vec{\beta}\times\vec{D})$ must be positive. Similarly, if $\mathcal{C}_R<0,$ then at least one must be negative. If $\mathcal{C}_R=0$, then we must have one of each sign, or all must be zero. It is therefore clear that a lack of existence may only occur if the coefficients $(\vec{\alpha}\times\vec{\beta}), (\vec{D}\times\vec{\alpha})$, and $(\vec{\beta}\times\vec{D})$ are all non-negative or all non-positive. To understand this condition geometrically  we rotate the problem so that the secant line lays on the x-axis, which implies $\vec{D}=\langle D , 0 \rangle$, and we consider the rays
\begin{align*}
L^{\alpha}&:=\tau\langle \alpha_1 , \alpha_2 \rangle, \text{ with $\tau\geq0$} \quad \text{and,}\\
L^{\beta}&:=\langle D, 0 \rangle-\nu\langle \beta_1 ,\beta_2 \rangle, \text{ with $\nu\geq0$}.
\end{align*}

\begin{lemma}\label{lem0}
Consider equation (\ref{AreaCon}) with $\vec{D}=\langle D , 0 \rangle$ for some $D>0$. Then, the coefficients of (\ref{AreaCon}) are all non-negative or all non-positive provided $L^{\alpha}$ and $L^{\beta}$ are on the same side of the x-axis and do not cross.
\end{lemma}
\begin{proof}
The two rays $L^{\alpha}$ and $L^{\beta}$ cross if the system
\begin{align*}
\tau\alpha_1&=D-\nu\beta_1, \quad \text{and,}\\
\tau\alpha_2&=-\nu\beta_2
\end{align*}
is satisfied for some positive values of $\tau$ and $\nu$. Suppose that both rays live above the x-axis, this implies that $\alpha_2>0$ and $\beta_2<0$, and therefore $(\vec{D}\times\vec{\alpha})>0$ and $(\vec{\beta}\times\vec{D})>0$. To ensure they do not cross we plug the second equation into the first and simplify to obtain
\begin{align*}
(\alpha_1\beta_2-\alpha_2\beta_1)\tau&=\beta_2D\\
\Rightarrow (\vec{\alpha}\times\vec{\beta})\tau&=-(\vec{\beta}\times\vec{D}).
\end{align*} 
We conclude that no positive solution $\tau$ exists if $(\vec{\alpha}\times\vec{\beta})>0$.  Therefore having both rays above the x-axis and not crossing yields all positive coefficients to (\ref{AreaCon}). Similarly having the rays below the x-axis and not crossing yields all negative coefficients.

One special case that may arise is if either $\alpha_2$ or $\beta_2$ is zero. Suppose without loss of generality that $\beta_2=0$. This means the second equation is always satisfied when $\tau=0$. To ensure the rays do not cross we thus require $\beta_1<0$. This implies $(\vec{\alpha}\times\vec{\beta})=-\alpha_2\beta_1$, which has the same sign as $(\vec{D}\times\vec{\alpha})$. Repeating the argument with $\alpha_2=0$ concludes the proof.
\end{proof}
\begin{remark}
We note that the coefficients of (\ref{AreaCon}) having the same sign does not imply a lack of admissible solutions. It does imply however that a solution will only exist if the provided area data has the appropriate sign.
\end{remark}

\begin{figure}[tb]
\begin{center}
\includegraphics[width=60mm,height=40mm]{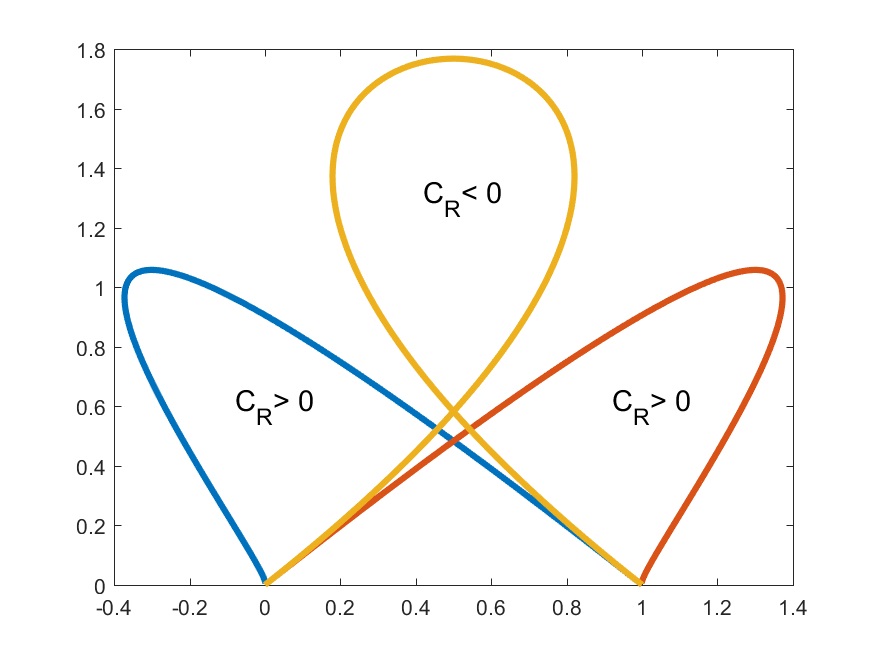}
\end{center}
\caption{ Solutions to (\ref{AreaCon}) for different area data. }
\label{AreasPlot}
\end{figure}

A sketch of how different areas can be obtained from the same endpoint and tangent data is shown in Figure \ref{AreasPlot}. In this example we have $\vec{\alpha}=\langle 1 , 1 \rangle$ and $\vec{\beta}=\langle 1, -1 \rangle$, which is an example where $L^{\alpha}$ and $L^{\beta}$ cross. From left to right in Figure \ref{AreasPlot} we have $(r_1,r_2)=(0.1,5),(r_1,r_2)=(5,5),(r_1,r_2)=(5,.1)$.  Within the context of our interpolation problem, if we fail to have existence of solutions to (\ref{AreaCon}), then a refinement of the mesh is required until the existence criteria is met.
\begin{remark}
We also note that in many application solutions with loops, such as the middle curve plotted in Figure \ref{AreasPlot}, are undesirable. In this case additional constraints can be placed on $r_1$ and $r_2$ to ensure such solutions are not permissible.
\end{remark}

We now turn our attention to choosing $r_1$ and $r_2$ satisfying (\ref{AreaCon}) possessing desirable convergence properties. In particular we are interested in finding $r_1$ and $r_2$ which minimize a prescribed distance between the planar curves $\vec{B}(t)$ and $\langle \gamma(s), \xi(s)\rangle$. Recall that given two parametric curves $\vec{P}_1:[0,1]\rightarrow\mathbb{R}^2$ and $\vec{P}_2:[0,1]\rightarrow\mathbb{R}^2$, the distance between the sets $P_1,P_2\in\mathbb{R}^2$ may be measured by the Hausdorff distance
\begin{equation}
d_{H}(P_1,P_2)=\max\Big(\sup_{x\in P_1}\inf_{y\in P_2}d(x,y),\sup_{y\in P_2}\inf_{x\in P_1}d(x,y)\Big).\label{Haus}
\end{equation}
Here we are seeking the asymptotic decay of $d_H(x,y)$ as the curve lengths tends to zero. In the context of our interpolation problem, for small enough curve length (after rotation if necessary), the target curve can be represented by the graph of a function $f(x)$.  Therefore, when investigating the accuracy of our interpolating parametric polynomial in this setting we may relax (\ref{Haus}) to the $L^{\infty}$ norm 
\begin{equation}
\|B_2(t)-f(B_1(t))\|_{L^{\infty}(t\in[0,1])}. \label{funcerror}
\end{equation} 
Taking the interpolation domain to be $x\in[0,h]$ we search for $r_1$ and $r_2$ which satisfy
\begin{equation}
||B_{h_2}(t)-f(B_{h_1}(t))||_{L^{\infty}([0,1])}=\mathcal{O}(h^4),
\end{equation}
where the subscript $h$ signifies that for each choice of $h>0$ the resulting B\'ezier interpolant, $\vec{B}(t)$, is generated with potentially different data.

To achieve the above result we employ the method utilized by de Boor et al. in \cite{DeBoore} by relying upon classical results for cubic Hermite polynomials $H(x)$. Matching the data $f(0),f'(0),f(h),f'(h)$ we have $||H(x)-f(x)||_{L^{\infty}([0,h])}=\mathcal{O}(h^4),$ with
\begin{equation}
||H(x)-f(x)||_{L^{\infty}([0,h])}\leq\max_{z\in[0,h]}\left|f^{(4)}(z)\right|\,\frac{h^4}{4!}.\label{HermErr}
\end{equation}
To do this, we require that the curve given by $\vec{B}_h(t)$ for $t\in[0,1]$ can be represented by the graph of a function $b_h(x)$.  We show for $0<r_1,r_2<3h
$ that $\frac{d}{dt}B_{h_1}(t)>0$ for $t\in (0,1)$.

 The first case we consider is when we can write $\vec{\alpha}=\langle 1, f'(0) \rangle$ and $\vec{\beta}=\langle 1, f'(h) \rangle$. Combining this with $\vec{A}=0$ and $\vec{D}=\langle h, f(h) \rangle$ in equation (\ref{Bezier2}) yields
\begin{equation}
B_{h_1}(t)=r_1(1-t)^2t+(3h-r_2)(1-t)t^2+ht^3. \label{Bh1}
\end{equation} 
To show $\frac{d}{dt}B_{h_1}(t)>0$ we consider it a function of three variables, $\frac{d}{dt}B_{h_1}(t)=g(r_1,r_2,t)$, given by
\begin{equation}
g(r_1,r_2,t)=t^2(3r_1+3r_2-6h)+t(-4r_1-2r_2+6h)+r_1. 
\end{equation}
We show that within the open set $\mathcal{R}=\big\{(r_1,r_2,t)\in\mathbb{R}^3 \big| (0,0,0)<(r_1,r_2,t)<(3h,3h,1)\big\}$ we have $g(r_1,r_2,t)>0$. We first observe that $g(r_1,r_2,t)$ has no critical points in $\mathcal{R}$, then, checking boundaries and corners of $\bar{\mathcal{R}}$, we have values in $[0,3h]$, implying that the interior has values within $(0,3h)$, and therefore $g(r_1,r_2,t)>0$ within the desired region. The cases when $\vec{\alpha}=\langle 0, 1 \rangle$ or $\vec{\beta}=\langle 0, 1 \rangle$ similarly lead to $\frac{d}{dt}B_{h_1}(t)>0$ for $(r_1,r_2,t)\in\mathcal{R}$. We note that if the data $\vec{\alpha}$ and $\vec{\beta}$ are given as unit vectors, $\vec{\alpha}= \frac{\langle 1 , f'(0) \rangle}{\sqrt{1+(f'(0))^2}}$ and $\vec{\beta}= \frac{\langle 1 , f'(h) \rangle}{\sqrt{1+(f'(h))^2}}$, then the above result becomes $\frac{d}{dt}B_{h_1}(t)>0$ for $0<r_1<3h\sqrt{1+(f'(0))^2}$ and $0<r_2<3h\sqrt{1+(f'(h))^2}$ by the same argument. 

Therefore, under the assumption that the target planar curve is the graph of a function $f(x)$, and that $(r_1,r_2,t) \in \mathcal{R}$, we may write
\begin{equation}
||B_{h_2}(t)-f(B_{h_1}(t))||_{L^{\infty}([0,1])}=||b_h(x)-f(x)||_{L^{\infty}([0,h])}
\end{equation}
where the graph of $b_h(x)$ for $x\in[0,h]$ yields the same curve as $\vec{B}_h(t)$ for $t\in[0,1]$.

Before proceeding we recall that the first spatial derivative of a planar curve $\langle x(t),y(t) \rangle$ is given by $\mathcal{D}^1(x(t),y(t))=\frac{y'(t)}{x'(t)}$, and that the $n^{th}$ spatial derivative is defined recursively by $\mathcal{D}^{n}(x(t),y(t))=\frac{\frac{d}{dt}\mathcal{D}^{n-1}(x(t),y(t))}{x'(t)}$.  Therefore, with $x(t)$ and $y(t)$ cubic polynomials, the $4^{th}$ spatial derivative simplifies to
\begin{equation}
\frac{x' \left(15 (x'')^2 y''-4x'x''' y''-6 y''' x' x''\right)-y' \left(15 (x'')^3-10 x''' x' x''\right)}{(x')^7}. \label{Fourth}
\end{equation}
Studying equation (\ref{Fourth}) allows us to prove the following Lemma. 

\begin{lemma}\label{lem1}
Suppose $\vec{B}_h(t)$ is interpolating a function $f(x)\in C^4([0,h])$ on $[0,h]$. Then, if $r_1=h+P_1(h)h^2$ and $r_2=h+P_2(h)h^2$, where $P_i(h)$ are polynomials in $h$, then $\vec{B}(t)$ is at least $4^{th}$ order accurate provided $P_1(h)+P_2(h)=\mathcal{O}(h^q)$ for any $q\geq 1$, or if $f''(0)=0$.
\end{lemma}
\begin{proof}
Without loss of generality we assume $\vec{B}_h(0)=\vec{0}$, as discussed previously. The given expression for $r_1$ and $r_2$ implies we can find an $h$ small enough to ensure $0<r_1,r_2<3h$, as eventually the leading linear term will dominate, regardless of the values of $P_i(h)$. Since we are taking $\vec{\alpha}=\langle 1, f'(0) \rangle$ and  $\vec{\beta}=\langle 1, f'(h) \rangle$, this guarantees that the graph of $\vec{B}_h(t)$ can be generated by some function $b_h(x)$ for each value of $h>0$ sufficiently small. The given interpolation problem tells us that $b_h(0)=f(0)=B_{h_2}(0)=0$, $b_h(h)=f(h)=B_{h_2}(1)$, $b'_h(0)=f'(0)=\frac{B_{h_2}'(0)}{B_{h_1}'(0)}$, and $b_h'(h)=f'(h)=\frac{B_{h_2}'(1)}{B_{h_1}'(1)}$ (The case where the slope at endpoints is infinite  can be ignored since further partitioning and rotation can be used).  For any given $h$ small enough, the interpolation error is therefore given by
\begin{equation*}
e(h)=||b_h(x)-f(x)||_{L^{\infty}[0,h]}.
\end{equation*}
Letting $H(x)$ be the cubic Hermite for $f(x)$ on $x\in[0,h]$, the triangle inequality yields
\begin{equation*}
||b_h(x)-f(x)||_{L^{\infty}[0,h]}\leq ||b_h(x)-H(x)||_{L^{\infty}[0,h]}+||H(x)-f(x)||_{L^{\infty}[0,h]}.
\end{equation*}
The second term is $\mathcal{O}(h^4)$ by definition of the Hermite polynomial of $f(x)$. We are thus left to show that $||b_h(x)-H(x)||_{L^{\infty}[0,h]}$ is $\mathcal{O}(h^4)$. Since $H(x)$ is also the cubic Hermite polynomial for $b_h(x)$, we apply inequality (\ref{HermErr}) to obtain
\begin{equation}
||b_h(x)-f(x)||_{L^{\infty}[0,h]}\leq \sup_{x\in(0,h)}\Big | \frac{d^4}{dx^4}b_h(x)\Big |\frac{h^4}{4!}. \label{ineq1}
\end{equation}
The proof will be complete if $\frac{d^4}{dx^4}b_h(x)$ is bounded as $h\rightarrow 0$. We rely on the property that for each $x\in[0,h]$ there exists a $t^*\in[0,1]$ such that $\displaystyle \frac{d^4}{dx^4}b_h(x)=\mathcal{D}^4(B_{h_1}(t^*),B_{h_2}(t^*))$. Using equation (\ref{Fourth}) we demonstrate that $\mathcal{D}^4(B_{h_1}(t),B_{h_2}(t))$ is bounded.  We obtain the desired result by replacing each term by its Taylor expansion, then showing the lowest term in the numerator and denominator are both $\mathcal{O}(h^7)$.

Since we are interpolating a function, after a shift to the origin, the terms from equation (\ref{Bezier2}) become $\vec{A}=\vec{0}$, $\vec{D}=\langle h,f(h) \rangle$, $\vec{\alpha}=\langle 1, f'(0) \rangle$, and $\vec{\beta}=\langle 1, f'(h) \rangle$.  Starting with
\begin{align*}
f(h)&=hf'(0)+\frac{h^2}{2}f''(0)+\mathcal{O}(h^3)\\
f'(h)&=f'(0)+hf''(0)+\frac{h^2}{2}f'''(0)+\mathcal{O}(h^3),
\end{align*}
 we obtain
 \begin{align*}
 B_{h_1}'(t)&=h+\mathcal{O}(h^2), \quad B_{h_1}''(t)=2h^2(P_1(h)(-2+3t)+P_2(h)(-1+3t))+\mathcal{O}(h^3),\\ 
 B_{h_1}'''(t)&=6h^2(P_1(h)+P_2(h))+\mathcal{O}(h^3), \\
  B_{h_2}'(t)&=f'(0)h+\mathcal{O}(h^2),\\ B_{h_2}''(t)&=2f'(0)h^2(P_1(h)(-2+3t)+P_2(h)(-1+3t))+f''(0)h^2 +\mathcal{O}(h^3),\\ 
 B_{h_2}'''(t)&=6f'(0)h^2(P_1(h)+P_2(h))+\mathcal{O}(h^3).
 \end{align*}
 In fact, we can simplify our computations further by rewriting the second component derivatives in terms of the first component derivatives, yielding
 \begin{align*}
 B_{h_2}'(t)&=f'(0)B_{h_1}'(t)+\mathcal{O}(h^2),\\
 B_{h_2}''(t)&=f'(0)B_{h_1}''(t)+f''(0)h^2+\mathcal{O}(h^3),\\
 B_{h_2}'''(t)&=f'(0)B_{h_1}'''(t)+\mathcal{O}(h^3).
 \end{align*}
Plugging this into equation (\ref{Fourth}), we obtain
\begin{align*}
15x'(x'')^2y''&=15f'(0)B_{h_1}'(t)(B_{h_1}''(t))^2\left(B_{h_1}''(t)+\frac{f''(0)}{f'(0)}h^2\right)=\mathcal{O}(h^7), \\
-4(x')^2y''x'''&=-4f'(0)(B_{h_1}'(t))^2B_{h_1}'''(t)\left(B_{h_1}''(t)+\frac{f''(0)}{f'(0)}h^2\right) +\mathcal{O}(h^7),\\
-6(x')^2y'''x''&=-6f'(0)(B_{h_1}'(t))^2B_{h_1}''(t)B_{h_1}'''(t) +\mathcal{O}(h^7),\\
-15y'(x'')^3&=-15f'(0)(B_{h_1}'(t))(B_{h_1}''(t))^3=\mathcal{O}(h^7),\\
10y'x'x''x'''&=10f'(0)(B_{h_1}'(t))^2B_{h_1}''(t)B_{h_1}'''(t) +\mathcal{O}(h^7).
\end{align*}
Summing all terms above results in exactly
\begin{align*}
-4f''(0)(B_{h_1}'(t))^2B_{h_1}'''(t)h^2+\mathcal{O}(h^7)=-24f''(0)h^6(P_1(h)+P_2(h))+\mathcal{O}(h^7).
\end{align*}
Therefore, if $P_1(h)+P_2(h)=\mathcal{O}(h^q)$, for $q\geq 1$, or if $f''(0)=0$, the resulting numerator is at least $\mathcal{O}(h^7)$, which completes the proof.
\end{proof}

This result gives some direction on how to choose $r_1$ and $r_2$ to ensure convergence. In fact, it shows the existence of an entire class of cubic B\'ezier interpolants which are $4^{th}$ order accurate or better.  However, if we want to prove that a choice of $P_i(h)$ can yield $5^{th}$ order accuracy or better, then another approach for measuring the error is required. 

In the following Lemma we show that there exists a class of $P_i(h)$ which are $5^{th}$ order accurate or better, provided the curvature doesn't vanish at the endpoints.  Later we show that choices of $P_i(h)$ which satisfy the area constraint are within the class of $5^{th}$ order accurate cubic B\'ezier curves.

\begin{lemma}\label{lem2}
Suppose $\vec{B}_h(t)$ is interpolating a function $f(x)\in C^4([0,h])$ on $[0,h]$ with $f''(0)\neq 0$. Then, if $r_1=h+P_1h^3$ and $r_2=h+P_2(h)$, where $P_1 \in \mathbb{R}$, and $\displaystyle P_2(h)=-\frac{f''''(0)+24f''(0)P_1}{24f''(0)}h^3+\mathcal{O}(h^4)$ we have $\vec{B}_h(t)$ is at least $5^{th}$ order accurate.
\end{lemma}

\begin{proof}
Suppose our B\'ezier polynomial $\vec{B}_h(t)=\langle B_{h_1}(t),B_{h_2}(t)\rangle$, satisfying $B_{h_1}'(t)>0$ on $t\in(0,1)$, is interpolating a function $f(x)$ on $x\in[0,h]$, with $f''(0)\neq 0$.  Later we will see that for small enough $h>0$ and given a reasonable choice of $P_1$ that $B_{h_1}'(t)>0$ on $t\in(0,1)$, but for now we leave this as an assumption. To obtain the desired result we work directly with the $L^{\infty}$ error,
\begin{equation}
||B_{h_2}(t)-f(B_{h_1}(t))||_{L^{\infty}(t\in[0,1])}. \label{exacterror}
\end{equation}
Without loss of generality we continue to assume $f(0)=0$, which implies $B_{h_1}(0)=0$.  We begin by studying the Taylor expansion of $f(B_{h_1}(t))$ centred about $t=0$,
\begin{align}
f(B_{h_1}(t))&=f'(0)B'_{h_1}(0)t+\Big(f''(0)(B'_{h_1}(0))^2+f'(0)B''_{h_1}(0)\Big)\frac{t^2}{2} \nonumber \\
&+\Big(f'''(0)(B'_{h_1}(0))^3+3f''(0)B''_{h_1}(0)B'_{h_1}(0)+f'(0)B'''_{h_1}(0)\Big)\frac{t^3}{6} \nonumber \\ &+\Big( f''''(0)(B'_{h_1}(0))^4+3f''(0)(B''_{h_1}(0))^2+4f''(0)B'_{h_1}(0)B'''_{h_1}(0)+6f'''(0)B''_{h_1}(0)(B'_{h_1}(0))^2    \Big)\frac{t^4}{24} \nonumber\\
&+\mathcal{O}(t^5), \label{fB1}
\end{align}

where $B_{h_1}(t)=(h+P_1h^3)t-(P_2(h)+2P_1h^3)t^2+(P_1h^3+P_2(h))t^3$ given by (\ref{Bezier2}) using the data $\vec{\alpha}=\langle 1 , f'(0) \rangle$, $\vec{\beta}=\langle 1 , f'(h) \rangle$, $\vec{D}=\langle h, f(h) \rangle$, with $r_1=h+P_1h^3$ and $r_2=h+P_2(h)$.\\

Plugging this into (\ref{fB1}) yields
\begin{align*}
f(B_{h_1}(t))&=(h+P_1h^3)f'(0)t+\Big(f''(0)(h+P_1h^3)^2-2f'(0)(P_2(h)+2P_1h^3)\Big)\frac{t^2}{2}\nonumber\\
&+\Big(f'''(0)(h+P_1h^3)^3-6f''(0)(P_2(h)+2P_1h^3)(h+P_1h^3)+6f'(0)(P_1h^3+P_2(h))\Big)\frac{t^3}{6} \nonumber\\
 &+\mathcal{O}(t^4).
\end{align*}
Dropping terms which are $\mathcal{O}(h^5)$ or higher  results in
\begin{align}
f(B_{h_1}(t))&=(h+P_1h^3)f'(0)t+\Big(f''(0)(h^2+2P_1h^4)-2f'(0)(P_2(h)+2P_1h^3)\Big)\frac{t^2}{2}\nonumber\\
&+\Big(f'''(0)h^3-6f''(0)(hP_2(h)+2P_1h^4)+6f'(0)(P_1h^3+P_2(h))\Big)\frac{t^3}{6}\nonumber\\
 &+\Big(f''''(0)h^4+24f''(0)(P_1h^4+P_2(h)h)\Big)\frac{t^4}{24}+\mathcal{O}(t^5).\label{fB2}
\end{align}
We proceed to show for any choice of $P_1\in\mathbb{R}$ that taking \\$ P_2(h)=-\frac{f''''(0)+24f''(0)P_1}{24f''(0)}h^3+\mathcal{O}(h^4)$ yields $||B_{h_2}(t)-f(B_{h_1}(t))||_{L^{\infty}(t\in[0,1])}=\mathcal{O}(h^5).$ 

Again using our data for $\vec{\alpha},\vec{\beta}, \vec{D}, r_1$ and $r_2$, equation (\ref{Bezier2}) yields

\begin{align*}
B_{h_2}(t)&=(h+P_1h^3)f'(0)t+\Big(3f(h)-2hf'(0)-hf'(h)-f'(h)P_2(h)-2f'(0)P_1h^3\Big)t^2\\
&+\Big(hf'(0)+hf'(h)-2f(h)+f'(0)P_1h^3+f'(h)P_2(h)\Big)t^3.
\end{align*}
Taylor expanding $f(h)$ and $f'(h)$ up to $\mathcal{O}(h^5)$ and dropping terms which are $\mathcal{O}(h^5)$ or higher, $B_{h_2}(t)$ simplifies to

\begin{align}
B_{h_2}(t)&=(h+P_1h^3)f'(0)t+\Big(f''(0)\frac{h^2}{2}-f''''(0)\frac{h^4}{24}-(f'(0)+hf''(0))P_2(h)-2f'(0)P_1h^3\Big)t^2\nonumber \\
&+\Big(f'''(0)\frac{h^3}{6}+f''''(0)\frac{h^4}{12}+f'(0)P_1h^3+(f'(0)+hf''(0))P_2(h)\Big)t^3. \label{Bh2}
\end{align}

Using (\ref{Bh2}) and (\ref{fB2}) we can compute $||B_{h_2}(t)-f(B_{h_1}(t))||_{L^{\infty}(t\in[0,1])}$. Going term by term we have
\begin{align*}
|B_{h_2}(t)-f(B_{h_1}(t))|&\leq \Big|(h+P_1h^3)-(h+P_1h^3)\Big|f'(0)t\\
&+\Big|-f''(0)(hP_2(h)+P_1h^4)-f''''(0)\frac{h^4}{24} \Big|t^2\\
&+\Big|f''''(0)\frac{h^4}{12}+f''(0)(2hP_2(h)+2P_1h^4) \Big|t^3\\
&+\Big|f''''(0)\frac{h^4}{24}+f''(0)(P_1h^4+P_2(h)h)\Big|t^4+\mathcal{O}(h^5)+\mathcal{O}(t^5)
\end{align*}
Setting $ P_2(h)=-\frac{f''''(0)+24f''(0)P_1}{24f''(0)}h^3+\mathcal{O}(h^4)$ we obtain the desired result that\\ $|B_{h_2}(t)-f(B_{h_1}(t))|=\mathcal{O}(h^5)$.
\end{proof}

\begin{remark}
We note that the $\mathcal{O}(t^5)$ terms do not contribute any terms lower than $h^5$, because each time we  differentiate a term of the form \\$f^{(n_1)}(0)(B'_{h_1}(0))^{n_2}(B''_{h_1}(0))^{n_3}(B'''_{h_1}(0))^{n_4}$, for $n_i\in\mathbb{N}$ the lowest power of $h$ must increase since $B'_{h_1}(0)=\mathcal{O}(h),B''_{h_1}(0)=\mathcal{O}(h^2)$ and $B'''_{h_1}(0)=\mathcal{O}(h^3)$.
\end{remark}
With Lemmas \ref{lem1} and \ref{lem2} in hand, we may now state our main result.
\begin{theorem}\label{thm1}
Parametric area-preserving cubic B\'ezier curves, satisfying $B_1'(0)=r_1=h+Ph^3$, for $P\in \mathbb{R}$, and $B_1'(1)=r_2=\frac{6(10\mathcal{C}_R-r_1 (\vec{D}\times\vec{\alpha}))}{r_1(\vec{\alpha}\times\vec{\beta})+6(\vec{\beta}\times\vec{D})}$, are $5^{th}$ order accurate provided the curvature does not vanish at the endpoints. In the zero curvature case, the interpolation is $4^{th}$ order accurate.
\end{theorem}
\begin{proof}
We begin by proving the non-zero curvature case, which simply breaks down into showing that setting $r_1=h+Ph^3$ and $r_2=\frac{6(10\mathcal{C}_R-r_1 (\vec{D}\times\vec{\alpha}))}{r_1(\vec{\alpha}\times\vec{\beta})+6(\vec{\beta}\times\vec{D})}$ satisfies the conditions from Lemma \ref{lem2}. Again, since we are seeking the error as $h\rightarrow 0$, we will assume that we are interpolating a function $f(x)$ on $[0,h]$. Inputting the data from $f(x)$ into the given equation for $r_2$ yields
\begin{align}
r_2&=\frac{6\left(10\left(\int_{0}^{h}f(x)\text{d}x-\frac{h f(h)}{2}\right) -(h+Ph^3)\left(\langle h,f(h) \rangle \times \langle 1,f'(0) \rangle \right)\right) }{(h+Ph^3)\left(\langle 1,f'(0) \rangle\times\langle 1,f'(h) \rangle\right)+6\left(\langle 1,f'(h) \rangle\times\langle h,f(h)\rangle\right)}\nonumber \\
&=\frac{6\left(10\left(\int_{0}^{h}f(x)\text{d}x-\frac{h f(h)}{2}\right) -(h+Ph^3)\left(hf'(0)-f(h)\right)\right) }{(h+Ph^3)\left(f'(h)-f'(0)\right)+6\left(f(h)-hf'(h)\right)}. \label{r2}
\end{align}

Plugging in the Taylor expansion for each term in (\ref{r2}), we obtain
\begin{equation}
r_2=\frac{-2f''(0)h^3-\frac{3}{2}f'''(0)h^4+(3f''(0)P-\frac{1}{2}f''''(0))h^5+\mathcal{O}(h^6)}{-2f''(0)h^2-\frac{3}{2}f'''(0)h^3+(f''(0)P-\frac{7}{12}f''''(0))h^4+\mathcal{O}(h^5)},
\end{equation} 
which, for small enough $h>0$ is equivalent to
\begin{equation}
r_2=h-\frac{f''''(0)+24f''(0)P}{24f''(0)}h^3+\mathcal{O}(h^4).
\end{equation}
Therefore, by Lemma \ref{lem2}, we have our result.

In the case that $f''(0)=0$, we follow a similar approach. Taking a Taylor expansion in $x$ of each term in $r_2$ we obtain
\begin{equation}
r_2=\frac{-\frac{3}{2}f'''(0)h^4-\frac{1}{2}f''''(0)h^5+\mathcal{O}(h^6)}{-\frac{3}{2}f'''(0)h^3-\frac{7}{12}f''''(0)h^4+\mathcal{O}(h^5)},
\end{equation} 
which, for small enough $h>0$, is equivalent to 
\begin{equation*}
r_2=h-\frac{f''''(0)}{18f'''(0)}h^2 +\mathcal{O}(h^3).
\end{equation*}  Therefore, by Lemma \ref{lem1}, since $r_1=h+Ph^3$ and $r_2=h+P_2(h)h^2$, for $P_2(h)$ a polynomial and $f''(0)=0$, we have that the curvature vanishing case is $4^{th}$ order accurate.
\end{proof}
A simple computation leads us to the following Corollary on optimality.

\begin{cor}\label{cor1}
There does not exist a $6^{th}$ order cubic B\'ezier area-preserving interpolating polynomial.
\end{cor}
\begin{proof}
Suppose we allow the constant $P_1$ to be a function of $h$, yielding $r_1=h+P_1(h)h^3$, with $r_2=\frac{6(10\mathcal{C}_R-r_1 (\vec{D}\times\vec{\alpha}))}{r_1(\vec{\alpha}\times\vec{\beta})+6(\vec{\beta}\times\vec{D})}$ to guarantee area preservation. Applying a Taylor expansion on the error (\ref{exacterror}) we obtain
\begin{equation*}
||B_{h_2}(t)-f(B_{h_1}(t))||_{L^{\infty}(t\in[0,1])}=\frac{(5f'''(0)f''''(0)-2f''(0)f'''''(0))(4t-1)t^2}{480f''(0)}h^5 + \mathcal{O}(h^6,t^4).
\end{equation*}
Observing that the $h^5$ term does not contain a $P_1(h)$ implies the only way to obtain $6^{th}$ order would be to have $P_2(h)\approx \frac{1}{h}$, to allow the $h^6$ term to cancel out the $h^5$ term.  This is equivalent to setting $r_1=h+P_1(h)h^2$, however, this leads to a complicated $h^5$ term containing the variable $t$. Therefore since we require $P_1(h)$ to not vary with $t$, no order higher than $h^5$ is possible.\\
\end{proof}

The following Corollary proves that if we are unable to provide an exact area, but instead an approximate area up to some order $\mathcal{O}(h^5)$, that the $L^{\infty}$ error drops to $\mathcal{O}(h^4)$.

\begin{cor}\label{cor2}
If the prescribed area is an approximation of the true area with order $\mathcal{O}(h^5)$, then the area-preserving B\'ezier interpolation discussed above satisfies $||B_{h_2}(t)-f(B_{h_1}(t))||_{L^{\infty}(t\in[0,1])}=\mathcal{O}(h^4)$.
\end{cor}
\begin{proof}
Suppose the prescribed area is not exact with error $\mathcal{O}(h^5)$. This changes equation (\ref{AreaCon}) to
\begin{equation}
\frac{r_1r_2}{60}(\vec{\alpha}\times\vec{\beta} )+\frac{r_1}{10}(\vec{D}\times\vec{\alpha})+\frac{r_2}{10}(\vec{\beta}\times\vec{D})=\mathcal{C}_R+Mh^5+\mathcal{O}(h^6),\label{AreaCon2}
\end{equation}
for $h$ small enough, where $M\in\mathbb{R}$ is some constant independent of $h$. Solving for $r_2$ yields
\begin{equation*}
 r_2=\frac{6(10(\mathcal{C}_R+Mh^5+\mathcal{O}(h^6))-r_1 (\vec{D}\times\vec{\alpha}))}{r_1(\vec{\alpha}\times\vec{\beta})+6(\vec{\beta}\times\vec{D})}.
\end{equation*}
Repeating the steps from Theorem \ref{thm1}, with $r_1=h+Ph^3$, we now obtain
\begin{equation*}
r_2=h-\frac{f''''(0)+24f''(0)P+720M}{24f''(0)}h^3+\mathcal{O}(h^4).
\end{equation*}
This results in $r_1=h+\mathcal{O}(h^3)$ and $r_2=h+\mathcal{O}(h^3)$ which satisfies the conditions for Lemma \ref{lem1}, implying $4^{th}$ order accuracy.
\end{proof}



\section{Numerical results}\label{SecNum}

\subsection{Preliminary Considerations}\label{Prelim}
In this section we present several examples illustrating the effectiveness of exact area-preserving cubic B\'ezier interpolation. For a few of these examples we employ two methods. First we use the standard area-preserving method, with $r_1=h+Ph^3$, where $P$ is chosen a priori. We also employ an optimized method, which we call the optimized area-preserving method, where $r_1$ and $r_2$ are chosen to numerically minimize the error by extracting additional information from the target function. The Figures presented in this section will refer to the standard area-preserving method as A-P, and the optimized area-preserving method as O-A-P.  For the standard scheme we assign $P$ to be $P_{avg}$, given by

\begin{equation}
P_{avg}=\frac{1}{h^3}\left[\frac{1}{2}\left(h+\frac{6(10\mathcal{C}_R-h(\vec{\beta}\times\vec{D}))}{h(\vec{\alpha}\times\vec{\beta})+6(\vec{D}\times\vec{\alpha})}\right) - h\right].\label{Pavg}
\end{equation}
We arrive at $P_{avg}$ by taking the average of $r_1=h$ and\\ $r_1=\frac{6(10\mathcal{C}_R-h(\vec{\beta}\times\vec{D}))}{h(\vec{\alpha}\times\vec{\beta})+6(\vec{D}\times\vec{\alpha})}$, which arises by solving the area equation for $r_1$ instead of $r_2$ and simply setting $r_2=h$. A similar computation as in the proof of Theorem \ref{thm1}, shows that $\frac{6(10\mathcal{C}_R-h(\vec{\beta}\times\vec{D}))}{h(\vec{\alpha}\times\vec{\beta})+6(\vec{D}\times\vec{\alpha})}=h-\frac{f''''(0)}{24f''(0)}h^3+\mathcal{O}(h^4)$. Plugging this into the equation for $P_{avg}$ we see that $P_{avg}=\mathcal{O}(1)$ as desired.

The standard scheme described above is an efficient $5^{th}$ order method which, in most cases, performs remarkably well.  However, there are choices of $r_1$ which lead to greater precision, but they are problem specific and require an iterative process to compute. In many applications this additional precision may not be necessary, however it is important to note that more optimal choices of $r_1$ and $r_2$ can be obtained.

We therefore seek solutions to the optimization problem
\begin{equation}
\min_{(r_1,r_2)\in\mathcal{S}}||B_{2}(t,r_1,r_2(r_1))-f(B_{1}(t,r_1,r_2(r_1)))||_{L^{\infty}(t\in[0,1])},\label{OptNorm}
\end{equation}
where $\mathcal{S}$ is the set of all $(r_1,r_2)>0$ resulting in a cubic B\'ezier curve which can be represented by a function $b(x)$. We note that an exact representation of $\mathcal{S}$ is easily obtained, however the square $0<r_1,r_2\leq 3h$ does sufficiently well.

We implement the optimized method by discretizing $0<r_1\leq3h$ into a fine grid $0<r_{1_1}<r_{1_2}<\dots<r_{1_n}=3h$, then compute $r_2(r_{1_i})$. For each pair $(r_{1_i},r_2(r_{1_i}))\in\mathcal{S}$ we approximate the norm (\ref{OptNorm}) then choose the minimizer amongst all candidates.  We note that if the target curve is a parametric curve $\langle \gamma(s) \, , \, \xi(s) \rangle$, then we replace the $L^{\infty}$ norm (\ref{OptNorm}) with the Hausdorff distance (\ref{Haus}). 

In the remaining subsections we will illustrate our theoretical findings through numerical examples. We  present convergence in the $L^{\infty}$ norm, but do not include any discussion of the error in the area, as, by construction, it was found to be machine precision in all examples. 
\subsection{Example 1: Unit Circle}
Our first example is repeating the first example in section 4 of \cite{DeBoore}, interpolating the unit circle. We instead interpolate the semi circle as our scheme is rotation invariant, and measure the $L^{\infty}$ error as we partition the semi circle into subintervals with decreasing length. Beginning with 2 sub intervals, or 4 points over the entire circle up to 32 points over the entire circle, or partitioning the semi circle into 16 subintervals. At each step we compare the $L^{\infty}$ error obtained by our area-preserving schemes versus the (non-area preserving) curvature matching scheme in \cite{DeBoore}. The results are presented in Table \ref{Tab1}.
\begin{table}[!ht]
\begin{center}
 \begin{tabular}{||c | c | c | c ||} 
 \hline
 Number of points & Curvature Matching \cite{DeBoore} & Area-Preserving & Optimized Area-Preserving \\ [0.5ex] 
 \hline\hline
 4 & 1.4$\times 10^{-3}$ & 2.9$\times 10^{-3}$ & 2.6$\times 10^{-4}$ \\ 
 \hline
 8 & 2.1$\times 10^{-5}$ & 5.7 $\times 10^{-5}$ & 4.5$\times 10^{-6}$  \\
 \hline
 16 & 3.2$\times 10^{-7}$ & 1.0 $\times 10^{-6}$ & 8.2$\times 10^{-8}$  \\
 \hline
 32 & 4.9$\times 10^{-9}$ & 1.6$\times 10^{-8}$ & 3.3$\times 10^{-9}$ \\
 \hline 
\end{tabular}
 \vspace{.1 cm}
 \caption{$L^{\infty}$ error of area-preserving versus curvature matching.\label{Tab1}}
\end{center}
\end{table}

The results show that the curvature matching and standard area-preserving methods remain close throughout the experiment, with the curvature matching method obtaining two to three times better accuracy for each test. This is somewhat surprising as the curvature method is a $6^{th}$ order accurate. We expect that if the experiments were continued to a finer partition we would see a larger discrepancy in accuracy appear. The optimized method however is significantly more accurate than the others with between 2 and 5 times more accuracy than the curvature matching method through sixteen points. At 32 points the $6^{th}$ order curvature method closes the gap and falls only slightly short of the optimized area-preserving method. Figure \ref{Circle convergence2} confirms that the convergence is indeed $5^{th}$ order, agreeing with the statement of Theorem \ref{thm1}.

\begin{figure}[!ht]
\begin{center}
\includegraphics[width=110mm,height=90mm]{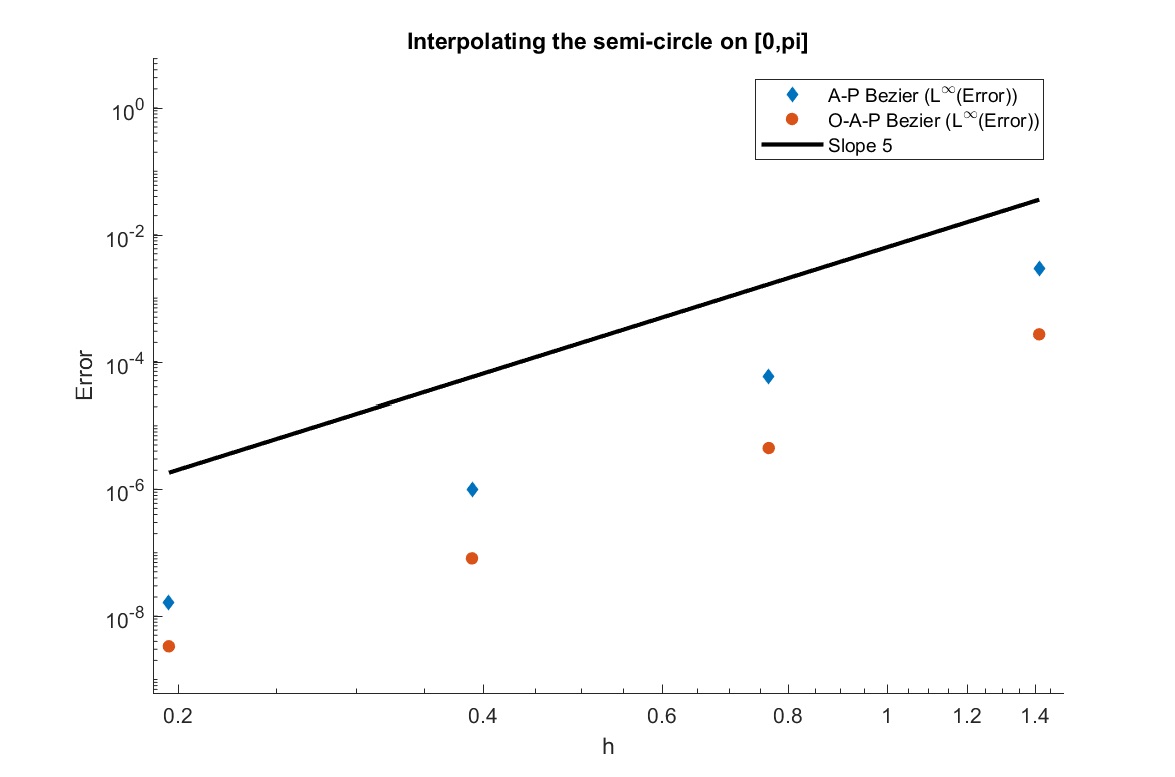}
\end{center}
\caption{Illustrating $5^{th}$ order convergence when interpolating the unit circle.}
\label{Circle convergence2}
\end{figure}

\subsection{Vanishing Curvature}
The following example illustrates the statement of Theorem \ref{thm1} when we have vanishing curvature. The target function is $f(x)=\sin(x)+3x^4-4x^3+x$ for $x\in[0,h]$, and the resulting convergence using $P_{avg}$ defined in section \ref{Prelim} is given in Figure \ref{VanishCurv}. As predicted we obtain $4^{th}$ order accuracy, the same as would be obtained by employing the method of de Boor et al. As we are only illustrating the $4^{th}$ order accuracy of the method in the presence of vanishing curvature, we do not discuss the optimized approach in this problem.

\begin{figure}[!ht]
\begin{center}
\includegraphics[width=110mm,height=90mm]{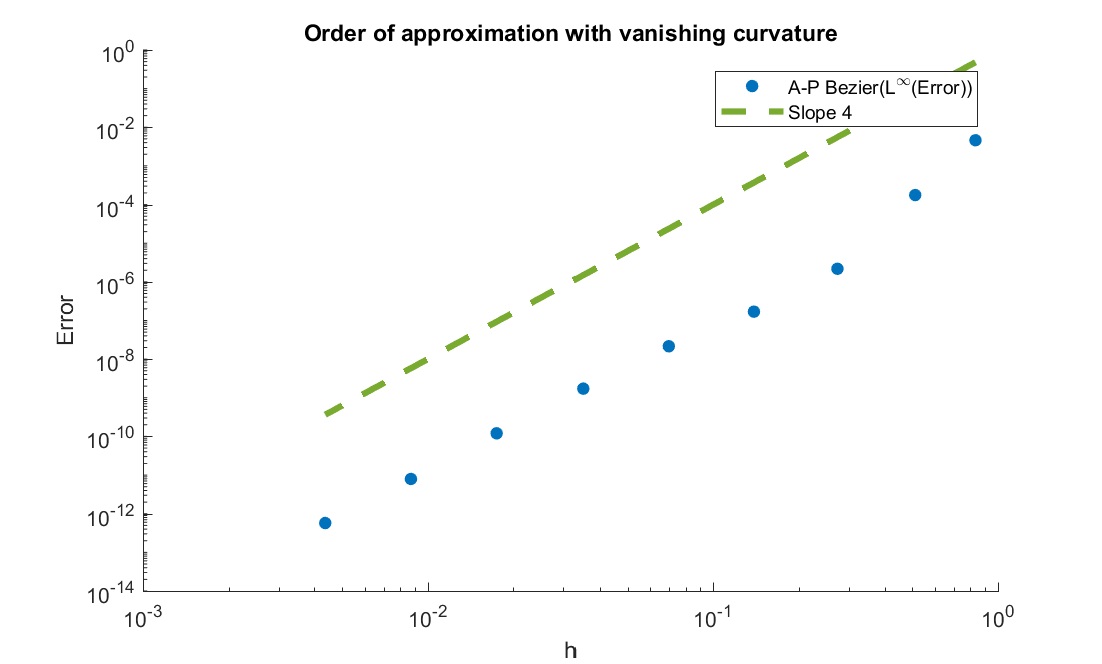}
\end{center}
\caption{Example of $4^{th}$ order accuracy with vansihing curvature.}
\label{VanishCurv}
\end{figure}

\subsection{Comparison to the CVE method}
To get a sense of how the standard area-preserving geometric interpolation method compares to other non-local $G^1$ methods, we repeat the last example from \cite{jaklicCVE}, the method which minimizes an approximate curvature variation energy.  In this example we are interpolating the parametric curve 
\begin{equation}
\vec{f}(t)=\langle (t^3-t+1)\sin(t),t\cos(t) \rangle \quad \text{for $t\in[0,1]$}. \label{JaklicFunc}
\end{equation}

The function (\ref{JaklicFunc}) is broken up into two subintervals $[0,0.3678]$ and $[0.3678,1]$ in the first example, and $[0,0.48]$ and $[0.48,1]$ in the second. The plot of the resulting interpolations are shown in Figure \ref{Jaklic}. Since (\ref{JaklicFunc}) is a parametric function, we estimate the error by approximating the Hausdorff distance by partitioning the interpolant and $\vec{f}(t)$ and computing (\ref{Haus}) discretely. The results are given in Table \ref{Tab2}. We note that the Hausdorff errors for the CVE method were not included in \cite{jaklicCVE}, however it is clear by inspection of Figure 5 and 6 in \cite{jaklicCVE} that the area-preserving method obtained significantly higher accuracy.
%
\begin{table}[!ht]
\begin{center}
 \begin{tabular}{||c | c ||} 
 \hline
 Interpolation intervals & Standard Area-Preserving \\ [0.5ex] 
 \hline\hline
 [0,0.3678] and [0.3678,1] & 6.9$\times 10^{-4}$  \\ 
 \hline
 [0,0.48] and [0.48,1] & 2.4 $\times 10^{-3}$  \\

 \hline 
\end{tabular}
 \vspace{.1 cm}
 \caption{Approximate Hausdorff error of the standard area-preserving method corresponding to the tests in Figure \ref{Jaklic}. \label{Tab2}}
\end{center}
\end{table}

\begin{figure}[!ht]
\begin{center}
\includegraphics[width=60mm,height=60mm]{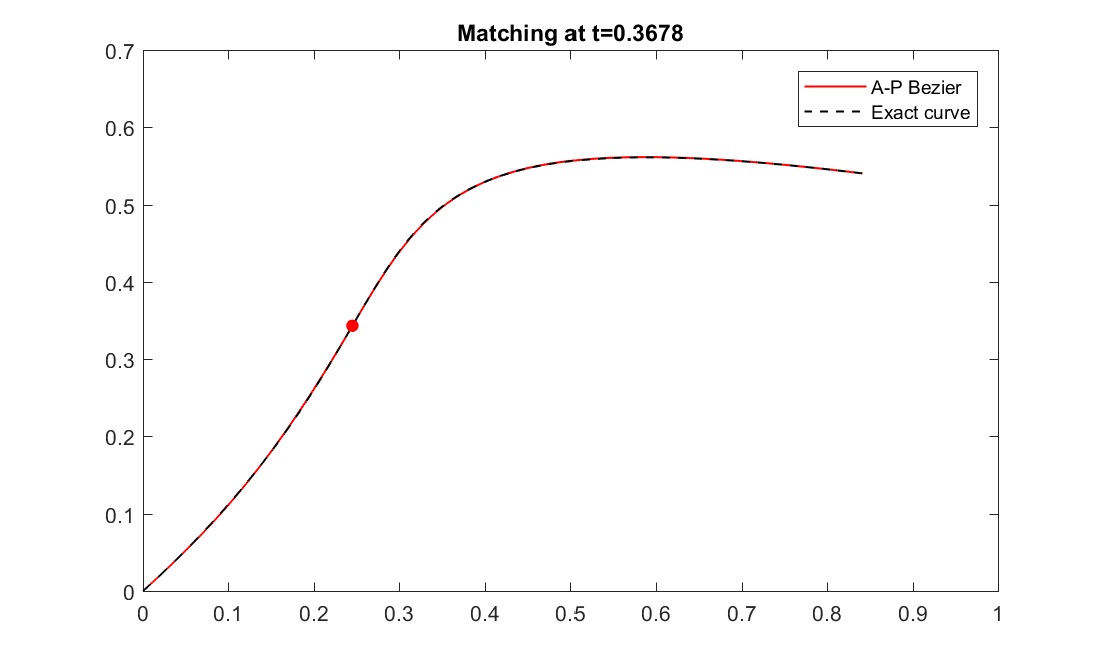}
\includegraphics[width=60mm,height=60mm]{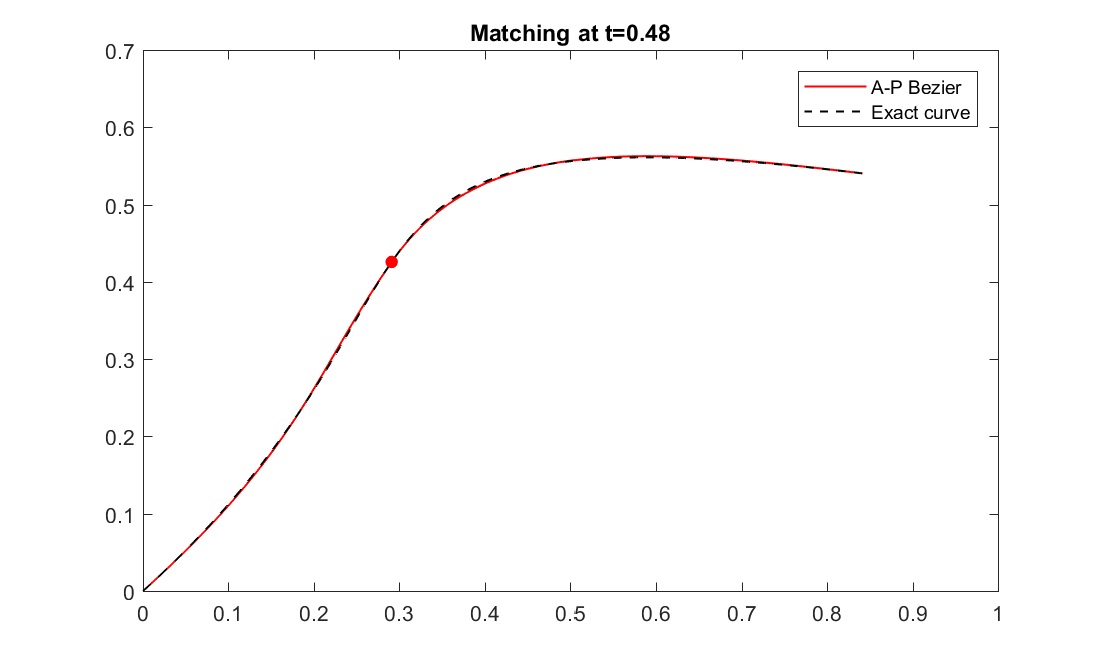}
\end{center}
\caption{Two area-preserving $G^1$ polynomials interpolate (\ref{JaklicFunc}) joining at $t=0.3678$ on the left panel and $t=0.48$ on the right panel.}
\label{Jaklic}
\end{figure}

\subsection{Optimized Area-Preservation}\label{SecOptBez}
In this next example we showcase the accuracy that may be gained by selecting optimized $r_1$ and $r_2$. We interpolate the function $f(x)=4x(x-0.5)(x-1)e^x$ on the interval $x\in[0,1]$. Figure \ref{OptimizedBezier} shows $f(x)$, the standard cubic Hermite polynomial interpolant (see \cite{Bartel}), the standard area-preserving cubic B\'ezier with $P=0$ and then the optimized area-preserving B\'ezier. We note that Figure \ref{OptimizedBezier} appears to only have 3 curves because the optimized curve is nearly identical to the target function $f(x)$. 

\begin{figure}[!ht]
\begin{center}
\includegraphics[width=100mm,height=60mm]{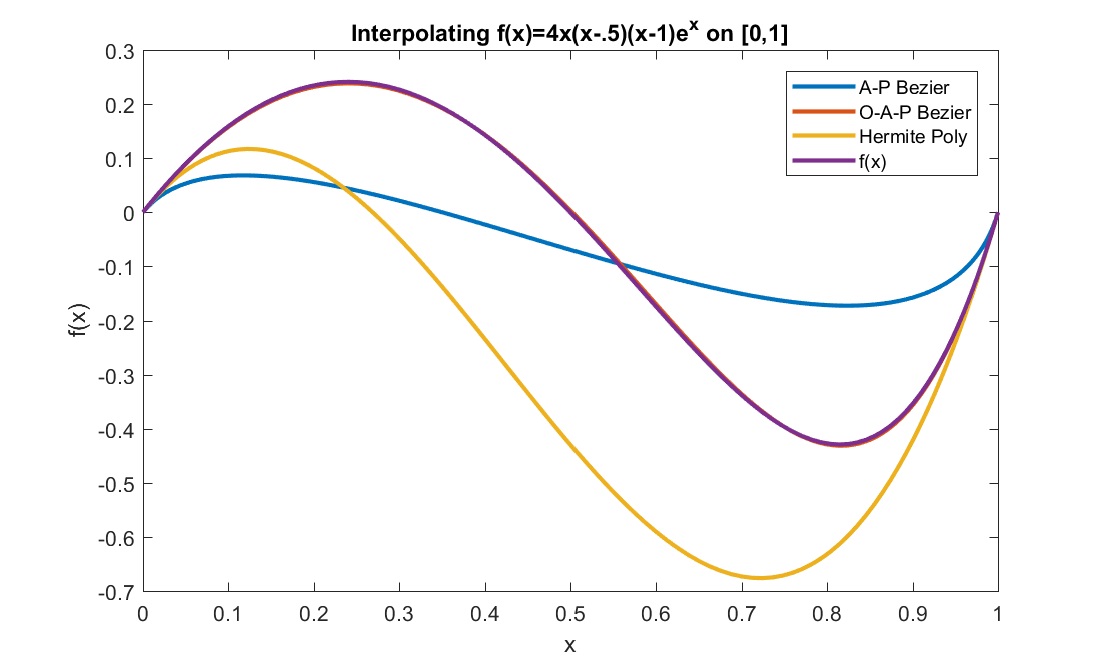}
\end{center}
\caption{Interpolation of $f(x)=4x(x-0.5)(x-1)e^x$.}
\label{OptimizedBezier}
\end{figure}

\subsection{Piecewise Interpolation I}

In the next example we construct a piecewise interpolation of the function $f(x)=(x+1)e^{x}-1$ on the interval $[0,1]$. We partition the domain into finer and finer subintervals, compute the interpolants on each subinterval, then compute the maximum $L^{\infty}$ error over all interpolants. Figure \ref{NoInflection} shows the $L^{\infty}$ error convergence for the standard cubic Hermite polynomial interpolant, the standard area-preserving cubic B\'ezier interpolation, and the optimized area-preserving cubic B\'ezier interpolation.

\begin{figure}[!ht]
\begin{center}
\includegraphics[width=100mm,height=60mm]{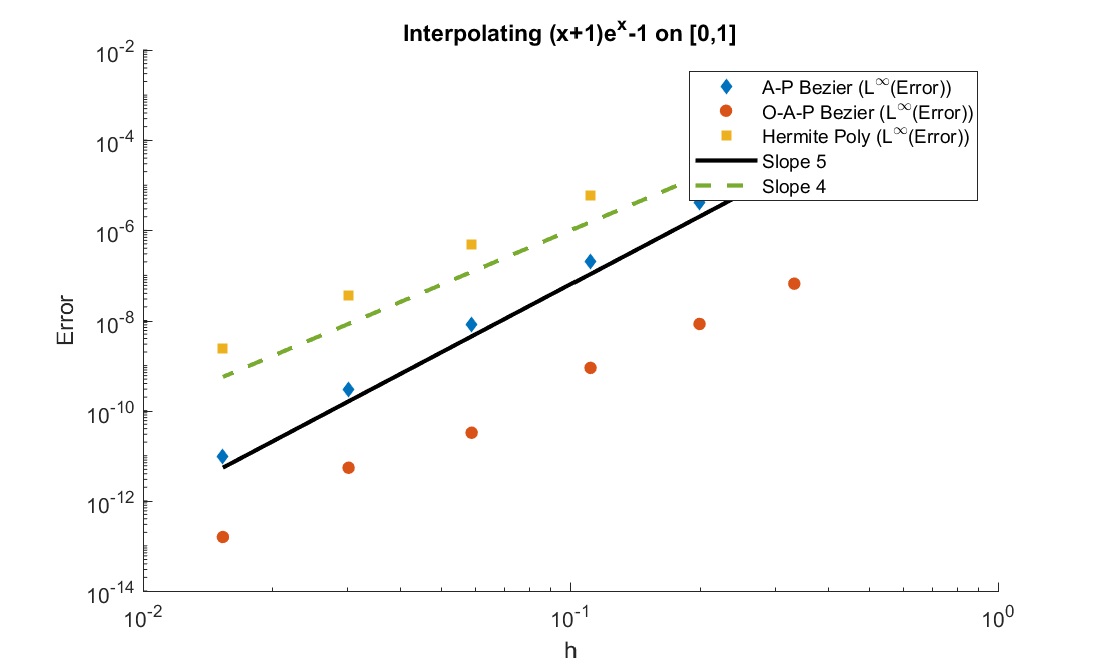}
\end{center}
\caption{$5^{th}$ order convergence of standard area-preserving method and optimized method.}
\label{NoInflection}
\end{figure}

Figure \ref{NoInflection} showcases the $5^{th}$ order convergence of the standard area-preserving method with $r_1=h$, shown in blue, and the $4^{th}$ order convergence of the Hermite polynomial. The dramatic improvement of the optimized curve is also visible as in Figure \ref{OptimizedBezier}, but the convergence is less consistent.
\subsection{Piecewise Interpolation II}
The last example has the same set up as the previous, but now the target function is $f(x)=x^2(1-x)e^x$. The significance of this example is that $f(x)$ has an inflection point, therefore by Theorem \ref{thm1} we expect to obtain $4^{th}$ order convergence. The results are shown in in Figure \ref{Inflection}. We see that both the optimized and unoptimized methods are $4^{th}$ order accurate by observing the last few data points. The inconsistency of the optimized convergence is simply a result of where the grid points land relative to the inflection point and areas with high curvature.

\begin{figure}[!ht]
\begin{center}
\includegraphics[width=100mm,height=60mm]{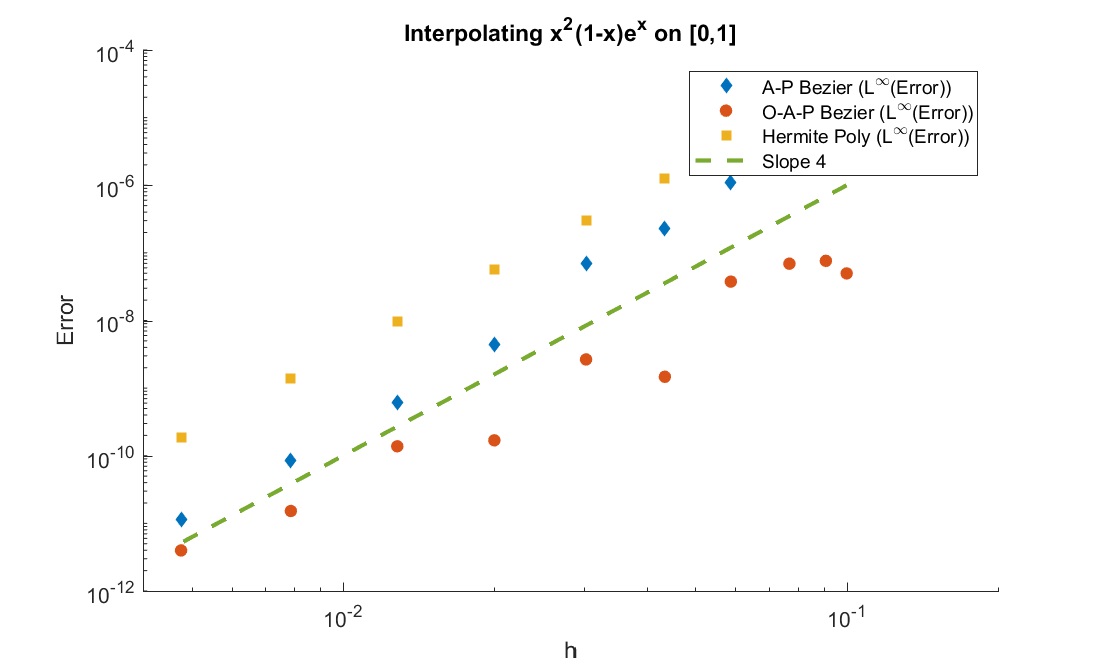}
\end{center}
\caption{$4^{th}$ order convergence of standard area-preserving method and optimized method when an inflection point is present within interpolation domain.}
\label{Inflection}
\end{figure}

\subsection{Scalar Conservation law}\label{Conservation}
Thus far we have yet to apply the area-preserving cubic B\'ezier interpolation framework to a problem where the conservation of area is vital to the resulting curve. Such an interpolation problem is standard in numerical methods for conservation laws, where conservation of the studied quantity is fundamental to the dynamics. As discussed briefly in \cite{LeVeque} and \cite{LevFinite}, weak solutions to scalar conservation laws in one space dimension may be obtained by utilizing the equal area principle. As shown in Figure \ref{Proj Sketch}, the application of the equal area principle requires one to find the location of a vertical line which partitions the overturned curve into two regions with equal area. The corresponding weak solution, as shown in the right panel of Figure \ref{Proj Sketch}, is discontinuous with a jump at the location of equal area line.

\begin{figure}[!ht]
\begin{center}
\includegraphics[width=40mm,height=30mm]{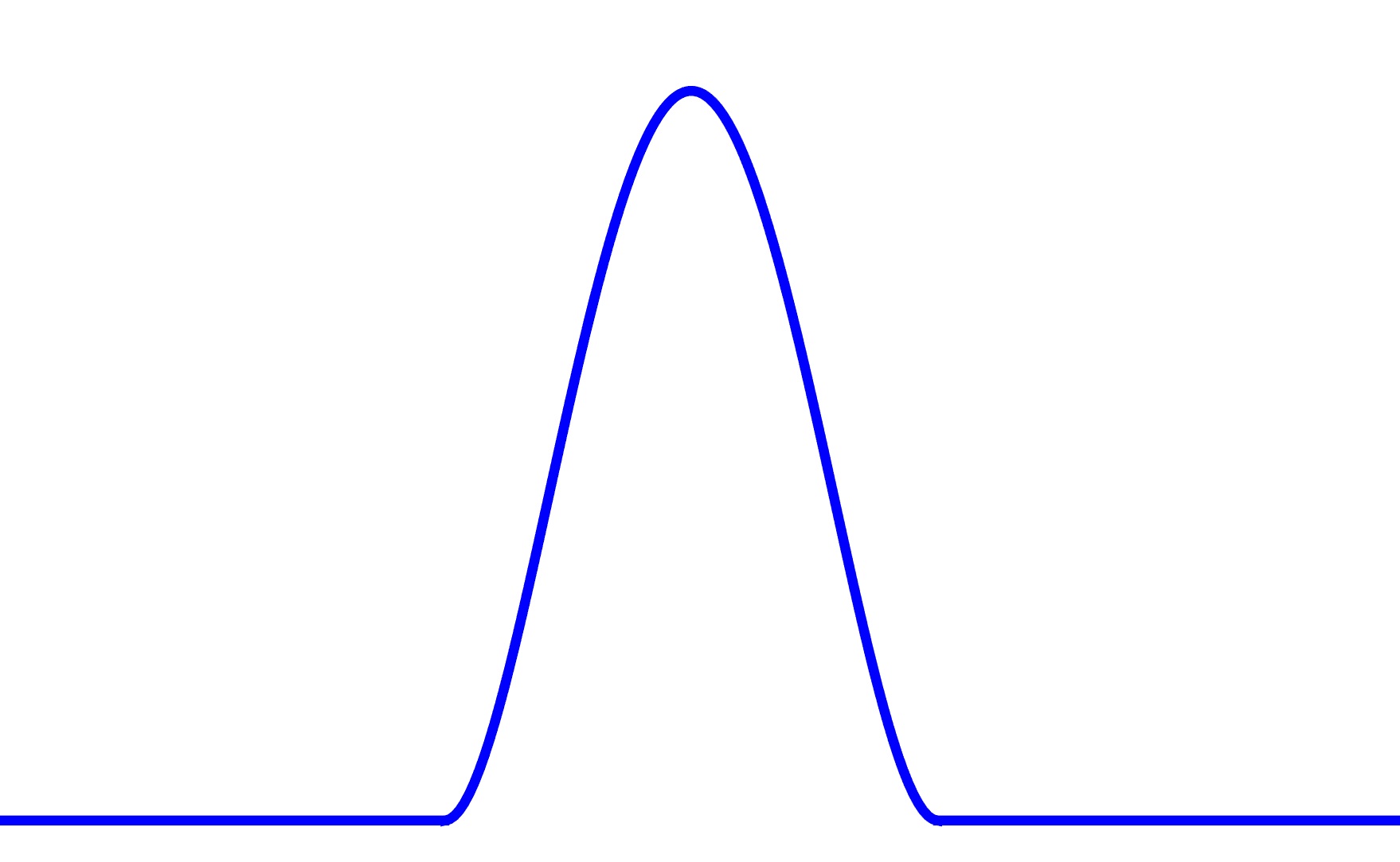}
\includegraphics[width=40mm,height=30mm]{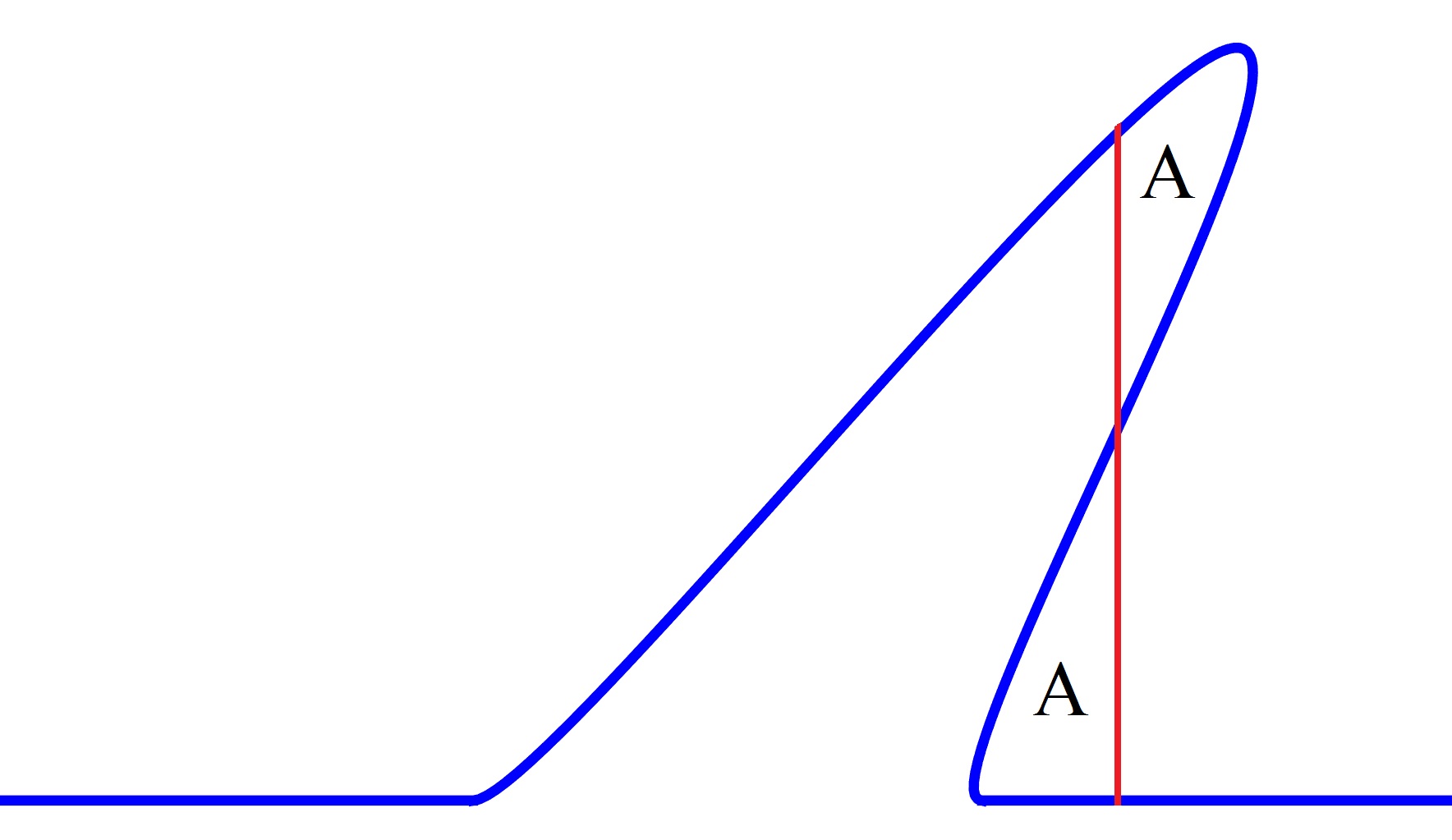}
\includegraphics[width=40mm,height=30mm]{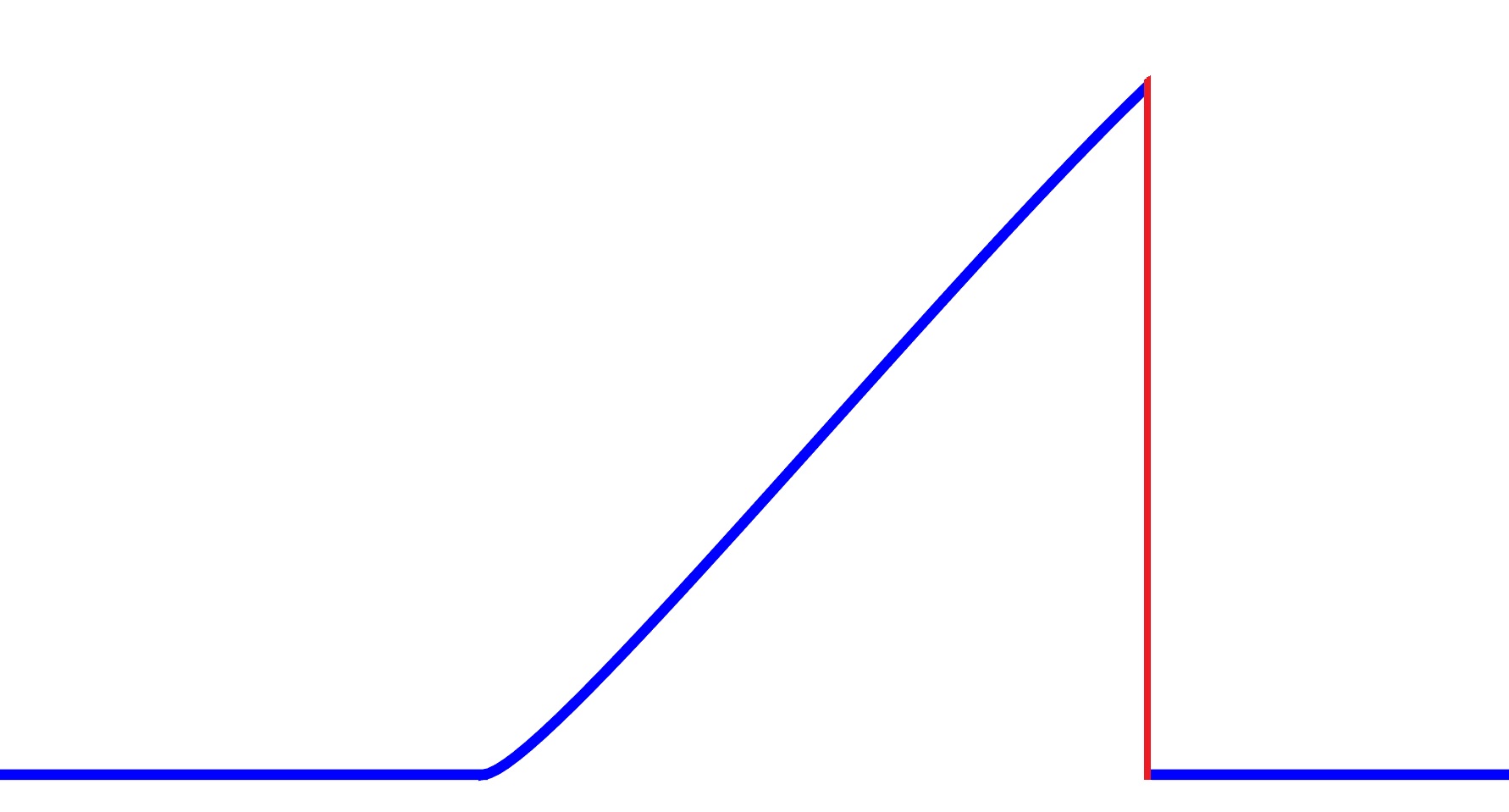}
\end{center}
\caption{ An illustration of the equal area principle. }
\label{Proj Sketch}
\end{figure}

From a numerical standpoint, the equal area principle is seldom leveraged as it requires the approximation of a multi-valued curve. To bypass this, for example in \cite{Seibold}, linear interpolation is employed.  Although we leave the analysis and discussion of a method utilizing area-preserving cubic B\'ezier interpolation for another paper, it is clear based on Figure \ref{Proj Sketch} that such an interpolation framework is ideal for the application of the equal area principle.   We note that the data required to implement area-preserving cubic B\'ezier interpolation may be obtained exactly by studying the characteristic equations of the corresponding conservation law.   
\section{Discussion}\label{Discussion}

In this paper we set out to design a new cubic B\'ezier interpolation framework which exactly preserves area while maintaining high accuracy. Theorem \ref{thm1} demonstrates that, provided we select $r_1$ correctly, matching the prescribed area with a geometric cubic Hermite polynomial yields $5^{th}$ order accuracy, one order higher than the standard geometric cubic Hermite. This order is optimal for area-preserving cubics as shown in Corollary \ref{cor1}. In section 3 we verify numerically the statements of Theorem \ref{thm1} by obtaining $5^{th}$ order accuracy when the curvature is non-vanishing and $4^{th}$ order otherwise. The numerical comparisons with the methods of \cite{DeBoore} and \cite{jaklicCVE} show that the our area-preserving method is competitive or superior in the $L^{\infty}$ or Hausdorff error, with the area-preserving method offering the additional benefit of being conservative. The optimized method discussed in section \ref{Prelim} and \ref{SecOptBez} shows that further improvements to accuracy can be obtained when an iterative process to find an optimal choice of $r_1$ and $r_2$ is employed. Finally in section \ref{Conservation} we briefly discuss how area-preserving geometric interpolation can be applied to conservations laws through the use of the equal area principle. In the near future we look forward to further investigating the application to conservation laws along with gradient augmented level set methods, such as those discussed in \cite{Nave1,Nave2}.

\bibliographystyle{amsplain}
\bibliography{references}

\end{document}